\documentclass[10pt]{amsart}
\usepackage{latexsym, amssymb, amscd, amsfonts,pstricks,enumitem,amsmath, young}
\usepackage{latexsym, amssymb, amscd, amsfonts,pstricks,enumitem,amsmath}
\usepackage{amsmath,amsthm,amssymb,mathrsfs}
\usepackage{amsmath, amsfonts, amssymb,amsthm}
\usepackage{amstext}
\usepackage{mathrsfs}
\usepackage{tikz-cd}
\usepackage{fancyhdr}
\usepackage{soul}
\usepackage{verbatim}

\numberwithin{equation}{section}

\theoremstyle{plain}
\newtheorem{theorem}{Theorem}[section]

\newtheorem{def-thm}[theorem]{Definition-Theorem}

\newtheorem{definition}[theorem]{Definition}

\theoremstyle{definition}
\newtheorem{remark}[theorem]{Remark}

\newcommand{\sq}[1]{\ifx#1([\else\ifx#1)]%
  \else\message{invalid use of "sq"}\fi\fi}

\DeclareMathOperator{\mult}{mult}

\DeclareMathOperator{\Supp}{Supp}

\DeclareMathSymbol{\idot}{\mathbin}{operators}{`\.}
\allowdisplaybreaks
\makeatother

\begin{document}
 
\title[Families of Unit Equations and Exponential Diophantine Problems]{Families of Unit Equations and Exponential Diophantine Problems via Integral Points}
 
\author{Julie Tzu-Yueh Wang}
\address{Institute of Mathematics, Academia Sinica \newline
\indent No.\ 1, Sec.\ 4, Roosevelt Road\newline
\indent Taipei 10617, Taiwan}
\email{tywang@as.edu.tw}

\author{Zheng Xiao}
\address{Department of Mathematics, University of Colorado Boulder
  \newline
\indent  2300 Colorado Avenue \newline
\indent  Boulder, CO 80309 USA} 
\email{zheng.xiao@colorado.edu}

\thanks{2020\ {\it Mathematics Subject Classification}: Primary 32H30; Secondary 32Q45 and 30D35}
\thanks{The  first-named author was supported in part by Taiwan's NSTC grant  113-2115-M-001-011-MY3.}

\begin{abstract} 
This paper investigates the distribution of integral points on projective varieties via two distinct methods: the Ru-Vojta theorem and our higher-dimensional generalization of the Huang-Levin-Xiao inequalities. These approaches operate under distinct geometric conditions, specifically the transverse and proper intersections of boundary divisors.  Applying this framework, we prove  degeneracy results for the solution sets of two classes of one-parameter families of unit equations, differentiated by the degrees of their polynomial coefficients.
Finally, we extend previous greatest common divisor (GCD) estimates to derive new results for specific exponential Diophantine equations and   the distribution of digits in  $q$-adic representations.
 \end{abstract}

\maketitle
 \baselineskip=16truept

 \section{Introduction and Main Theorems}\label{sec:intro}
 \subsection{Introduction}
 In a recent breakthrough, Ru and Vojta established a significant generalization of Schmidt's Subspace Theorem.  
 \begin{theorem}[{ \cite[General Theorem (Arithmetic Part)]{RV20}}]\label{Ru-Vojta} Let $k$ be a number field and $M_k$ be the set of places on $k$. Let $S\subset M_k$ be a finite subset containing the archimedean places.  Let $X$ be a projective variety defined over $k$.
 Let $D_1,\hdots,D_q$  be effective Cartier divisors intersecting properly on $X$.
 Let  $\mathcal L$ be a big line sheaf on $X$.  Then for any $\epsilon>0$, there exists a   proper Zariski-closed subset 
 $Z\subset X$, independent of $k$ and $S$,  such that
\begin{align}
 \sum_{i=1}^q \beta_{\mathcal L, D_i} m_{D_i,S}(x)\le (1+\epsilon) h_{\mathcal L}(x)
\end{align}
holds for  all but finitely many $x$ in  $X(k)\setminus Z$.
\end{theorem}
In this statement, for any closed subscheme \(Y\subset X\), the constant $\beta_{\mathcal{L}, Y}$ is defined by  
\begin{equation*}
  \beta(\mathscr L, Y)
    = \liminf_{N\to\infty}
      \frac {\sum_{m\ge 1}h^0(\mathscr L^N\otimes {\mathscr I}_{Y}^{m})} {Nh^0(\mathscr L^N)},
\end{equation*}
where ${\mathscr I}_{Y}$ is the ideal sheaf defining $Y$. Additionally, $m_{D_i,S}(x)$ denotes  the proximity function and $h_{\mathcal L} (x)$ is the height function associated with $\mathcal L$, as defined in Section \ref{Preliminary1}.  Note that  for a big divisor $D$ on $X$, we will set $\mathcal L=\mathcal O_X(D)$ and adopt the simplified notation $h_{D}:=h_{\mathcal O_X(D)}$ and  $\beta(D, Y):= \beta(\mathcal O_X(D), Y)$.

The foundational techniques developed to prove Theorem~\ref{Ru-Vojta}, alongside its far-reaching applications, have become central to Diophantine approximation and geometry. Its relationship with the earlier work of McKinnon and Roth \cite{MR2015} was clarified by Ru and Wang \cite{RuWang2017}. The theorem was subsequently extended to the setting of subvarieties by Ru and Wang \cite{RuWang2022}, and independently by Vojta \cite{Voj2023}. Taking a different direction, Heier and Levin \cite{HL2021} obtained a related inequality in which the $\beta$-constants are replaced by Seshadri constants. As this circle of ideas continues to develop, it is expected to lead to a rich range of further applications.
More recently, Huang, Levin, and Xiao \cite{HLX}  proved the following refinement.
\begin{theorem}[{\cite[Theorem 1.4]{HLX}}]\label{chainHLX}
   Let $X$ be a projective surface defined over a number field $k$.  Let $S$ be a finite set of places of $k$.  For each $v\in S$, let $D_v\supset Y_v$ be a regular sequence of nonempty closed subschemes of $X$. Let $A$ be a big divisor on $X$, and let $\epsilon >0$.   
 Then there exists a proper  Zariski-closed subset $Z\subset X$ such that  for all  $P\in X(k)\setminus Z$, we have
  \begin{align}\label{mainineq}
\sum_{v\in S} \Bigg(
  \beta(A, D_{v})\lambda_{D_{v},v}(P)  
 + \Big( \beta(A, Y_v ) -   \beta(A, D_{v}) \Big) \lambda_{Y_v,v }(P)
\Bigg)  
 <(1+\varepsilon) h_{A}(P). 
\end{align}
\end{theorem}
Here, $\lambda_{D_{v},v} $ and  $\lambda_{Y_v,v }$ are the local  Weil functions for $D_{v}$ and $Y_v$ respectively, with respect to $v$.  The definition of a regular sequence of nonempty closed subschemes of $X$ is given in Definition \ref{regular chain}.

 This result leads to several significant applications in \cite{HLX}, notably to the GCD problem, the study of integral points on certain affine surfaces, and the analysis of families of unit equations.

Theorem~\ref{chainHLX} was subsequently generalized by the authors in \cite{WX2025} to arbitrary projective varieties, although the subvarieties $Y_v$ are still assumed to be points. In the complex setting, this generalization enables the study of hyperbolicity for quasi-projective varieties whose boundary divisor consists of $n+1$ numerically parallel effective divisors on a complex projective variety of dimension $n$, assuming a non-empty intersection.   This result complements that of Ru and Wang in \cite{RW24}, where the divisors are required to have normal crossings. Furthermore, \cite{RW24} extends specifically the case established by Noguchi, Winkelmann, and Yamanoi \cite{NWY07} concerning the Green-Griffiths-Lang conjecture for $X=\mathbb P^n$ with $D$ consisting of $n+1$ normal crossings hypersurfaces. 

In this paper, we study the arithmetic applications of the number field counterpart of \cite{WX2025}. Notably, several of the applications derived in \cite{HLX} were previously investigated by Corvaja and Zannier \cite{CZ00, CZ2004, CZ2005, CZ2006, CZ2010} using different methods. It is worth emphasizing that the techniques developed in many of these foundational works were crucial, laying out groundwork for the eventual establishment of Theorem~\ref{Ru-Vojta}. With Theorem~\ref{Ru-Vojta} now firmly established, it is natural to revisit these earlier results through its lens. By adopting this theoretical approach, we not only clarify the ingenious methods originally used by Corvaja and Zannier, but also successfully generalize several of their results to higher dimensions. Finally, we compare these methodologies, highlighting the subtleties between the approach utilizing Theorem~\ref{Ru-Vojta} and the higher-dimensional version of Theorem~\ref{chainHLX} developed in \cite{WX2025}.
 
  \subsection{Main Theorems}\label{mainThm}
In this subsection, we first state the arithmetic counterpart to \cite[Theorem 1.9]{WX2025}, thereby generalizing Theorem~\ref{chainHLX}. We then explore an application concerning integral points, extending the surface case discussed in \cite[Theorem 1.5]{HLX}. Finally, we apply Theorem~\ref{Ru-Vojta} to obtain  a higher-dimensional extension of the result established by Corvaja and Zannier in \cite[Corollary 1.2]{CZ2006}, which, under different assumptions, serves to complement our aforementioned extension of \cite[Theorem 1.5]{HLX}.

To state our theorems, we first recall some definitions and conventions from \cite{HLX}. 


  \begin{definition}
    Let $D_1,\ldots, D_r$ be effective Cartier divisors on a projective variety $X$. We say that $D_1,\ldots, D_r$ intersect properly if for every nonempty subset $I\subset \{1,\ldots, r\}$ and every point $ P \in \Supp \cap_{i\in I}D_i$, the local defining equations $(f_i)_{i\in I}$ form a regular sequence in the local ring $\mathcal{O}_{X,P}$, where each $f_i$ is a local equation of $D_i$ at $P$.
\end{definition} 
 
 \begin{definition}\label{regular chain}
Let $Y_1\supset Y_2\supset \cdots\supset Y_m$ be closed subschemes of a projective variety $X$. We say that this chain forms  a {\it regular sequence of closed subschemes of $X$} if, for each $i\in \{1,\ldots, m\}$ and every point $P\in \Supp Y_i$, there exists a regular sequence $f_1,\ldots, f_i\in \mathcal{O}_{X,P}$ such that for $1\leq j\leq i$, the ideal sheaf of $Y_j$ is locally given by the ideal $(f_1,\ldots, f_j)$ in $\mathcal{O}_{X,P}$.
\end{definition}
\begin{remark}
    If divisors $D_1,\ldots, D_m$ intersect properly on $X$, then the chain of nested intersections $D_1\supset D_1\cap D_2\supset\cdots\supset \bigcap_{i=1}^m D_i$ forms a regular sequence of closed subschemes of $X$. Moreover, if $X$   is smooth and  $D_1,\ldots,D_m$ are in general position, then this chain is always a regular sequence.
    \end{remark}

The complex counterparts of the following two arithmetic theorems were established in \cite{WX2025}. Accordingly, we state the arithmetic results here without proof. These theorems extend Theorem~\ref{chainHLX} and \cite[Theorem 5.9]{HLX}, respectively.

\begin{theorem}[cf. {\cite[Theorems 1.8]{WX2025}}]\label{trungeneral}
   Let $X$ be a projective variety of dimension $n$ defined over a number field $k$.  Let $\mathcal M$ be a finite collection of  pairs $ (I,Y) $, where $I\subset\{1,\hdots,q\}$ has cardinality $n-1$, the divisors $D_i$ for $i \in I$ intersect properly, and the chain $\bigcap_{i \in I} D_i \supset Y$ forms a regular sequence.
Let $A$ be a big divisor on $X$.  
 Then, for every $\varepsilon>0$,  there exists a proper  Zariski-closed subset $Z\subset X$ such that 
  \begin{align}\label{mainineq}
\sum_{v\in S} \max_{(I,Y) \in \mathcal{M}}\Bigg(
&\sum_{i \in I}  \beta(A, D_{i})\lambda_{D_{i},v}(P)  
 + \Big( \beta(A, Y ) - \sum_{i \in I} \beta(A, D_i) \Big) \lambda_{Y,v }(P)
\Bigg)  \notag\\
&<(1+\varepsilon) h_{A}(P) 
\end{align}
 for all but finitely many $P\in X(k)\setminus Z$.
\end{theorem}

As an immediate consequence of Theorem \ref{trungeneral}, we arrive at Theorem \ref{mainTheorem2}, which addresses the distribution of $S$-integral points under specific numerical and intersection-theoretic hypotheses. 

\begin{theorem}[cf. {\cite[Theorems 1.4]{WX2025}}]\label{mainTheorem2}
Let $X$ be a  projective variety  of dimension $n$ over a number field $k$. 
Let $D_1, \ldots,D_{n+1}$ be effective Cartier  divisors on $X$ such that  there exist positive integers $a_1,\ldots, a_{n+1}$ such that $a_1D_1, \ldots, a_{n+1}D_{n+1}$ are all numerically equivalent to an ample divisor $D$.  
Assume that   any $n$-tuple of divisors among  $D_1, \ldots,D_{n+1}$ intersect properly, and that
\begin{align*}
\bigcap_{i=1}^{n+1}D_i\neq \emptyset.
\end{align*}
Suppose further that for every point 
 $Q\in  \bigcap_{i=1}^{n+1}D_i  $ and every subset $I \subset \{1,\ldots,n+1\}$, with $|I|=n$,  the following inequality holds:
 \begin{align}
\label{betacond}
\beta(D,(\bigcap_{i \in I}a_iD_i)_Q)>1.
\end{align}
Then  there exists a proper Zariski-closed subset $Z\subset X$ such that for any set $R\subset X(k)$ of $(\sum_{i=1}^{n+1}D_i, S)$-integral points   $ R\setminus Z$ is finite.
 \end{theorem}
We recall the following remark from \cite{WX2025}.
\begin{remark}\label{localcondition}
\begin{enumerate}
\item The condition \eqref{betacond} on the $\beta$-constant is sharp.
\item The inequality \eqref{betacond}  can be replaced by the following condition involving local multiplicities:
 \begin{align}\label{localmultip} 
   (D_i)^I_Q<\bigg(\frac{n}{n+1}\bigg)^n (D_i)^I, 
\end{align}
where $(D_i)^I_Q$ denotes the local intersection multiplicity of $D_i$ (for  $i\in I$) at the point  $Q$.  \end{enumerate}
\end{remark}

We now recall the notion of transverse intersection of divisors, which imposes a stronger geometric condition than that required in the theorem above.
 
Let $X$ be a projective variety and let $Y \subset X$ be a closed subscheme. 
For a point $P \in Y$, denote by $T_{Y,P}$ the Zariski tangent space of $Y$ at $P$.
\begin{definition}\label{transversal}
Let $D_1,\ldots,D_{n}$ be reduced and irreducible divisors on a  projective variety $X$ of dimension $n$.
Let $P\in \bigcap_{i=1}^{n} D_i$.
We say that $D_1,\ldots,D_{n}$ intersect transversally at $P$ if $X$ is smooth at $P$ and
\[
\operatorname{codim}_{T_{X,P}}
\left(\bigcap_{i=1}^n T_{{ D_i},P}\right)=n.
\]
\end{definition}
 
  In the setting of transverse intersections, Corvaja and Zannier established the degeneracy of integral points for certain configurations of non-normal crossing divisors on surfaces.

\begin{theorem}[{ \cite[Corollary 1.2]{CZ2006}}]\label{CZ2006}
    Let $X$ be a projective surface over a number field $k$. Let $D_1,D_2,D_3$ be distinct, effective, irreducible, and numerically equivalent divisors on $X$. Suppose that $D_1 \cap D_2 \cap D_3$ consists of a single point, at which each pair of the divisors $D_i$ intersects transversally, and that $D_i\cdot D_j >1$ for some $i\neq j$. Then no set of $(D_1+D_2+D_3,S)$-integral points in $X(k)$ is Zariski dense in $X$.
\end{theorem}

The following theorem provides a higher-dimensional analogue of the preceding result. Its proof applies Theorem~\ref{Ru-Vojta} to the blowup of $X$ along the intersection points of the $n+1$ divisors, while 
 Theorem~\ref{mainTheorem2} follows from a direct application of Theorem~\ref{trungeneral} to the variety $X$. 
 
  \begin{theorem}\label{mainTheoremCZrev}
Let $X$ be a  projective variety  of dimension $n$ over a number field $k$. 
Let $D_1, \ldots,D_{n+1}$ be reduced, irreducible, base point free divisors on $X$, all numerically equivalent to an ample divisor $D$.  
Assume that  
$\bigcap_{i=1}^{n+1}D_i $
is supported at a single point $Q$, and that
   any $n$-tuple among  $D_1, \ldots,D_{n+1}$ intersect  transversally at $Q$.
Suppose further that  
 $ D^n > 1$. Then  there exists a proper Zariski-closed subset $Z\subset X$ such that for any set $\mathcal R\subset X(k)$ of $(\sum_{i=1}^{n+1}D_i, S)$-integral points, the set   $\mathcal  R\setminus Z$ is finite.
 \end{theorem}
\begin{remark}\label{localcondition2}
Under assumptions of the theorem, the condition $ D^n > 1 $  is equivalent to the existence of an index set $I_0 \subset \{1,\ldots,n+1\}$ with $|I_0|=n$ such that the intersection $\bigcap_{i \in I_0} D_i$ contains more than one point. 
\end{remark}

Finally, we present a GCD-type result as an application of Theorem~\ref{trungeneral}. This result generalizes the surface case of \cite[Theorem 5.3]{HLX}; we omit the proof here, as its complex counterpart has already been established in \cite{WX2025}. For a comprehensive overview of various GCD estimates, we refer the reader to \cite{HLX}, which includes discussions on the work of Levin \cite{LevinGCD}, Levin and Wang \cite{LW}, Wang and Yasufuku \cite{WY}, and Huang and Levin \cite{HL}.

\begin{theorem}[cf. {\cite[Theorems 1.10]{WX2025}}]\label{MainThmgcd}
Let $D_1, \ldots,D_{n+1}$ be effective  divisors intersecting properly on a  projective variety $X$ of dimension $n$ over a number field $k$. Suppose that there exist positive integers $a_1,\ldots, a_{n+1}$ such that $a_1D_1, \ldots, a_{n+1}D_{n+1}$ are all numerically equivalent to an ample divisor $D$.  
Suppose that for some index set $I_0 \subset \{1,\ldots,n+1\}$ with $|I_0|=n$, the intersection $\bigcap_{i\in I_0}D_i$ contains more than one point. 
Let $S$ be a finite set of places of $k$ containing all the archimedean ones and let $\varepsilon > 0$. Then  there exists a proper Zariski-closed subset $Z\subset X$ such that for any set $R\subset X(k)$ of $(\sum_{i=1}^{n+1}D_i, S)$-integral points, we have 
\begin{align*}
h_{\bigcap_{i \in I_0}a_iD_i}(P) 
        \le \varepsilon h_D(P)  
\end{align*}
for all but finitely many points $P\in R\setminus Z$.
\end{theorem}
\subsection{Applications}
In this subsection, we provide applications for each of the main theorems established in Section~\ref{mainThm}.

First, we study general cases of  one-parameter family of unit equations as application of Theorem~\ref{mainTheorem2} and Theorem~\ref{mainTheoremCZrev}. Previous works on one-parameter families of unit equations have mainly focused on the specific form
\[f_1(t)u+f_2(t)v=f_3(t),\quad t\in\mathcal O_S, u,v\in \mathcal O_S^*, \]
where $f_1,f_2,f_3$ are polynomial over a number field $k$.
This equation was studied by Corvaja and Zannier (\cite{CZ2006}, \cite{CZ2010}) when $\deg f_1=\deg f_2=\deg f_3$,  by Levin (\cite{Levin06}) when $\deg f_1+\deg f_2=\deg f_3$, and under various other degree conditions in \cite{HLX}.

By applying Theorem~\ref{mainTheorem2}, we obtain the following extension of \cite[Theorem 1.6]{HLX}.
 
\begin{theorem}\label{unit1}
    Let $f_1,\ldots,f_{n+1} \in k[t]$ be non-constant polynomials of degrees $d_1,\ldots,d_{n+1}$ respectively,  without a common zero. Suppose that 
    \begin{equation}\label{degrees}
     \max_i (d_i+1) <   \frac{(n+1)^n}{(n+1)^n-n^n} \min_i d_i.
    \end{equation}
    Let  $V$  be the hypersurface  in $\mathbb A^{n+1}$ defined by the zero locus of  
   \begin{equation}\label{funiteq1}
        f_1(x_0)x_1 + \cdots + f_n(x_0)x_n=f_{n+1}(x_0).
    \end{equation} 
  Then there exists a proper Zariski closed subset $Z\subset V$    such that all but finitely many solutions $(t,u_1,\hdots,u_{n-1})\in \mathcal O_S\times  (\mathcal O_S^*)^n$ of  equation  \eqref{funiteq1}
are contained in $Z$.  
 \end{theorem}
Note that if the degrees $d_i$ are small and the dimension of $X$ is large, condition \eqref{degrees}  may fail to hold. However, by applying Theorem~\ref{mainTheoremCZrev}, we obtain the following extension of \cite[Theorem 4.1]{CZ2006} and \cite[Theorem 8]{CZ2010}. Notably, this result covers the case $d_1=\cdots=d_{n+1}=1$, which is left unaddressed by the preceding theorem.

\begin{theorem}\label{unit2}
Let $f_i\in k[t]$, $1\le i\le n+1$, be non-constant polynomials of the same degree $d$.
Let  $V$  be the hypersurface  in $\mathbb A^{n+1}$ defined by the zero locus of  
   \begin{equation}\label{funiteq}
        f_1(x_0)x_1 + \cdots + f_n(x_0)x_n=f_{n+1}(x_0).
    \end{equation} 
  Then there exists a proper Zariski closed subset $Z\subset V$    such that all but finitely many solutions $(t,u_1,\hdots,u_{n})\in \mathcal O_S\times  (\mathcal O_S^*)^n$ of the equation  \eqref{funiteq}
are contained in $Z$.  
 \end{theorem}

Second, we present two applications of Theorem~\ref{MainThmgcd} concerning perfect powers in exponential Diophantine equations and in digit representations. Exponential Diophantine equations involving perfect powers of polynomial values have been studied extensively in number theory. In particular, Corvaja and Zannier \cite{CZ00} studied equations of the form
\[
f(a^m,y)=b^n,
\]
where $f$ has rational coefficients and $a,b$ are positive integers with a nontrivial common factor. Related problems were later investigated in \cite{HLX} under different assumptions on the polynomial $f$.
Our first application of Theorem~\ref{MainThmgcd} considers exponential Diophantine equations of the form
\[
f(a_1^{m},\ldots,a_{\ell}^{m},y)=b^n,
\]
where $a_1,\ldots,a_{\ell},b>1$ and $\gcd(a_1,\ldots,a_{\ell},b) > 1$.   This extends the work of \cite{HLX}.
The precise statement of this result is given in Sections~\ref{applyGCDdiop}.
 
 The problem of representing perfect powers with few nonzero digits has become an influential topic in Diophantine analysis; we refer the reader to \cite{CZ2013} for further references and developments in this direction. In \cite{CZ2013}, Corvaja and Zannier proved that there exist only finitely many perfect powers in $\mathbb{N}$ whose binary expansion contains exactly four nonzero digits. A key ingredient in their approach is a result on the integer values of analytic power series evaluated at $S$-units \cite{CZ2005}, derived from the Subspace Theorem.  We apply Theorem~\ref{MainThmgcd} to study this problem. Our goal is to demonstrate how Theorem~\ref{MainThmgcd} can serve as a replacement for the aforementioned Diophantine ingredient in \cite{CZ2005}.
 The statement of this result is provided in Section~\ref{applyGCDdigit}.

Preliminary material on heights and the definition of integral points are provided in Section~\ref{Preliminary}. The proof of Theorem~\ref{mainTheoremCZrev} is given in Section~\ref{ProofmainTheoremCZ}, while the proofs concerning unit equations are deferred to Sections~\ref{Secunit1} and \ref{Secunit2}.

 \section{Preliminary}\label{Preliminary} 
 
\subsection{Number Fields and Heights}\label{Preliminary1} 
Let $k$ be a number field, $M_{k}$ the set of places of $k$ and $\mathcal{O}_{k}$ the ring of integers of $k$. For $v \in M_{k}$, let $k_{v}$ denote the completion of $k$ with respect to $v$. Throughout the paper, we normalize the absolute value $|\cdot |_{v}$ corresponding to $v\in M_{k}$ as follows: If $v$ is archimedean and $\sigma$ is the corresponding embedding $\sigma:k \to \mathbb{C}$, then for $x \in k^{*}$, $|x|_{v}=|\sigma(x)|^{[k_{v}:\mathbb{R}]/[k:\mathbb{Q}]}$; if $v$ is non-archimedean corresponding to a prime ideal $\mathscr{P}$ in $\mathcal{O}_{k}$ which lies above a rational prime $p$, then it is normalized so that $|p|_{v}=p^{-[k_{v}:\mathbb{Q}_{p}]/[k:\mathbb{Q}]}$. In this notation, we have the product formula:
$$\displaystyle \prod_{v \in M_{k}}|x|_{v}=1$$
for all $x \in k^{*}$.
 
Let $S$ be a finite set of places in $M_{k}$. The ring of $S$-integers and the group of $S$-units are denoted by $\mathcal{O}_{k,S}$ and $\mathcal{O}_{k,S}^{*}$ respectively.

 We now define (local) Weil functions following \cite[Section 9]{VojtaBook}. Let $X$ be a projective variety over a number field $k$, and let $D$ be a Cartier divisor on $X$. Let $v \in M_k$, and let $\mathbb{C}_v$ be the completion of the algebraic closure $\bar{k}_v$ of $k_v$. A Weil function for $D$ at the place $v$ is a function $\lambda_{D,v}: (X \setminus \Supp(D))(\mathbb{C}_v) \to \mathbb{R}$ that satisfies the following local condition. For each $x \in X$, there exist an open neighborhood $U$ of $x$, a nonzero regular function $f \in \mathcal{O}(U)$ such that $D|_U=(f)$, and a continuous, locally bounded function $\alpha: U(\mathbb{C}_v) \to \mathbb{R}$ satisfying
\begin{equation*}
    \lambda_{D,v}(x) = -\log|f(x)|_v + \alpha(x)
\end{equation*}
for all $x \in (U\setminus \Supp(D))(\mathbb{C}_v)$. 
The Weil function for $D$ is then defined as the map
\[ \lambda_D: \coprod_{v\in M_k} (X\setminus \Supp(D))(\mathbb{C}_v) \to \mathbb{R}. \]
We also recall that  an $M_k$-constant is a collection $(c_v)$ of constant $c_v\in\mathbb R$ for each $v\in M_k$, such that $c_v=0$ for almost all $v$.

Weil functions satisfy the following properties:

\begin{enumerate}
     \item[(a)] If $\lambda_1$ and $\lambda_2$ are Weil functions for Cartier divisors $D_1$ and $D_2$ on $X$, respectively, then $\lambda_1 + \lambda_2$ extends uniquely to a Weil function for $D_1 + D_2$.
    
    \item[(b)] If $\lambda$ is a Weil function for a Cartier divisor $D$ on $X$, and if $f : X' \to X$ is a morphism of $k$-varieties such that $f(X') \not\subset \text{Supp } D$, then $x \mapsto \lambda(f(x))$ is a Weil function for the Cartier divisor $f^*D$ on $X'$.
    
    \item[(c)] If $X = \mathbb{P}^n_k$, and if $D$ is the hyperplane at infinity, then the function
    \begin{equation*}
        \lambda_{D,v}([x_0 : \dots : x_n]) := -\log \frac{|x_0|_v}{\max\{|x_0|_v, \dots, |x_n|_v\}}
    \end{equation*}
    is a Weil function for $D$.
    
    \item[(d)] If both $\lambda_1$ and $\lambda_2$ are Weil functions for a Cartier divisor $D$ on $X$, then $\lambda_1 = \lambda_2 $ up to an $M_k$-constant.
    
    \item[(e)] If $D$ is an effective Cartier divisor and $\lambda$ is a Weil function for $D$, then $\lambda$ is bounded from below by an $M_k$-constant.
    
    \item[(f)] If $D$ is a principal divisor $(f)$, then $-\log |f|$ is a Weil function for $D$.
\end{enumerate}

The theory of Weil functions extends beyond divisors to arbitrary closed subschemes, following the framework developed by Silverman \cite{Sil1987}. Let $Z(X)$ denote the set of closed subschemes of $X$. For any $W, Y \in Z(X)$ and a place $v$, let $\lambda_{W,v}$ and $\lambda_{Y,v}$ be the Weil functions associated with $W$ and $Y$, respectively. We can then define a Weil function for the intersection $W \cap Y$, unique up to an $O(1)$ bound, as their pointwise minimum:
\begin{equation*}
    \lambda_{W\cap Y,v} = \min\{\lambda_{W,v}, \lambda_{Y,v}\}.
\end{equation*}
We refer the reader to \cite[Theorem 2.1]{HLX} for a concise summary of the properties of these functions.

Let $S \subset M_k$ be a finite set of places, and for each $v \in S$, let $\lambda_{D,v}$ be a Weil function for $D$ at $v$, then, up to $O(1)$, we define the proximity function $m_{D,S}(x)$ as
\begin{equation*}
    m_{D,S}(x):=\sum_{v \in S}\lambda_{D,v}(x).
\end{equation*}
Then we define the (global) height function. The height function $h_{D}(x)$ for points $x \in X(k)$ is defined, up to $O(1)$, as
\begin{equation*}
    h_D(x):=\sum_{v \in M_k}\lambda_{D,v}(x)
\end{equation*}
for any Weil function $\lambda$ for $D$.

\subsection{Integral Points}
We briefly recall the definition of integral points from \cite[Section 13]{VojtaBook}.

Let $k$ be a number field. A point $(x_1,\ldots,x_n) \in \mathbb{A}^n(k)$ is an integral point if all $x_i$ lie in $\mathcal{O}_k$. Let $S \subset M_k$ be a finite set of places containing archimedean ones, then $(x_1,\ldots,x_n)$ is an $S$-integral point if all $x_i$ lie in the ring $\mathcal{O}_{k,S}$ of $S$-integers.

Let $X$ be an affine variety over $k$. Then a set $R \subset X(k)$ is $S$-integral if there is a closed immersion $i: X \hookrightarrow \mathbb{A}^n_k$ for some $n$ and a nonzero element $a \in k$ such that for all $x \in R$ all coordinates of $i(x)$ lie in $(1/a)\mathcal{O}_{k,S}$. Hence we refer to integrality of a set of points.

Alternatively, the notion of integral points can be formulated in terms of Weil functions. Let $X$ be a complete variety over $k$, and let $D$ be an effective Cartier divisor on $X$. Then a set $R \subset X(k)$ is a $(D,S)$-integral set of points if no point $x \in R$ lies in the support of $D$ and there is a Weil function $\lambda_D$ for $D$ and an $M_k$-constant $(c_v)$ such that
\begin{equation*}
    \lambda_{D,v}(x) \leq c_v,
\end{equation*}
for all $x \in R$ and all places $v \not \in S$. Equivalently, if $R$ be a set of $(D,S)$-integral points, then
\begin{equation*}
    m_{D,S}(x)=\sum_{v \in S}\lambda_{D,v}(x)=h_D(x) +O(1).
\end{equation*}

 \section{Proof of Theorem \ref{mainTheoremCZrev}}\label{ProofmainTheoremCZ}

 \begin{proof}[Proof of Theorem \ref{mainTheoremCZrev}]
By assumption, the intersection 
$\bigcap_{i=1}^{n+1}D_i $
is supported at a single point $Q$, and 
   any $n$ divisors  among  $D_1, \ldots,D_{n+1}$ intersect  transversally at $Q$.
Let $\pi : \tilde{X}\to X$ be the blowup along $Q$ and let $E=\pi^{-1}(Q)$ denote the exceptional divisor.  
Since  $\mult_Q(D_i)=1$ for all $1\le i\le n+1$, the pullback of $D_i$ is given by 
 \begin{align}\label{pullbak}
\pi^*D_{i}=\tilde D_i+E ,
\end{align}
where $\tilde D_i$ denotes the strict transform of $D_i$.  
Moreover,  
$$
(-1)^{n-1}E^n=1.
$$
Since $D_i$ is ample and base point free, it follows that $\tilde D_i$ is nef.  Furthermore, $\tilde D_i$ is big because $D_i^n=D^n\ge 2$. 
Let 
$$
A=\ell(\tilde D_1+\hdots +\tilde D_{n+1})+E,
$$
where $\ell$ is a sufficiently large integer  to be determined later.  Then $A$ is big and nef.

\medskip

\noindent{\bf Claim:}
 $\beta(A,D_i)>\ell$ for $1\le i\le n+1$, and $\beta(A,E)>1$.

\medskip

Assuming the claim, write $\beta(A,D_i)=\ell+\alpha$ for $1\le i\le n+1$, and $\beta(A,E)=1+\delta$, where $\alpha$ and $\delta$ are positive numbers.  The transversality of $D_i$ at $Q$ ensures that strict transforms $\tilde{D}_i$ and $E$ intersect properly on $\tilde{X}$.  Applying  Theorem \ref{Ru-Vojta} to the divisors $\tilde D_i$, $1\le i\le n+1$, $E$ with respect to  $ A=\ell(\tilde D_1+\hdots +\tilde D_{n+1})+E$, and taking  $ \varepsilon =\min\{\frac{\alpha}{2\ell},  \frac{\delta}2\}$, we obtain a proper Zariski closed subset $\tilde Z$ of $\tilde  X$ such that for all $P\in \tilde X(k)\setminus \tilde Z$,
\begin{align}\label{onX2}
(\ell+\alpha)  \sum_{i=1}^{n+1}  m_{\tilde D_{i},S}(P)+  (1+\delta)m_{E,S}(P)
 <(1+ \varepsilon)  \sum_{i=1}^{n+1} \ell h_{\tilde D_{i}}(P)+(1+ \varepsilon) h_E(P).
\end{align}
Now suppose that $\pi(P)\in \mathcal R$, where $\mathcal R\subset X(k)$ is a set   of $(D_1+\cdots+D_{n+1},S)$-integral points.  Then  for all $ P\in \pi^{-1}(\mathcal R) \setminus \tilde Z$, we have
$$
m_{\tilde D_{i},S}(P)=  h_{\tilde D_{i}}(P)+O(1),\quad m_{E,S}(P)=  h_{E}(P) +O(1).
$$ 
 By our choice of $\varepsilon$, inequality \eqref{onX2} implies that 
$$
h_{\tilde D_{i}}(P) \le O(1),\quad h_{E}(P) \le O(1).
$$
  Consequently, we have $h_{D}(\pi(P)) \le O(1).$
Since $D$ is ample, Northcott's theorem yields the desired finiteness.

We now prove the claim.
From \eqref{pullbak}, we have 
\begin{align}\label{NA}
 A\equiv  \ell (n+1)\pi^*D-\big( \ell (n+1)- 1\big)E.
\end{align}
Let $N$ be a sufficiently large integer compared to $\ell$.
To estimate $\beta(A,E)$, we  compute
\begin{align*}
&h^0(\tilde X,NA-mE)=h^0(\tilde X,N\ell (n+1)\pi^*D-\big(N\ell (n+1)-N+m\big)E)\cr
&= h^0( X,N\ell (n+1)D)-\frac{\big(N\ell (n+1)-N+m\big)^n}{n!}+O(\big(N\ell (n+1)-N+m\big)^{n-1}).
\end{align*}
by \cite[Corollary  6.9]{Voj2023} and \cite[Lemma 3.10]{HLX}.  
Let $c\in \mathbb Q$ with $0<c<\sqrt[n]{D^n}$ and $cN\in\mathbb Z$.
For $m<cN$, we obtain
\begin{align*}
&n!\cdot  h^0(\tilde X,NA-mE) 
&= N^n\ell^n (n+1)^nD^n- \big(N\ell (n+1)-N+m\big)^n +O(N^{n-1}). 
\end{align*}
Summing over $0\le m\le cN$, we deduce
\begin{align*}
&n!\cdot \sum_{m=0}^{cN} h^0(\tilde X,NA-mE) \cr
&= cN^{n+1}\ell^n (n+1)^nD^n- \frac{1}{n+1}\big(\big(N\ell (n+1)-N+cN\big)^{n+1}-\big(N\ell (n+1)-N\big)^{n+1}\big) +O(N^{n }) \cr
&  \geq cN^{n+1}\ell^n (n+1)^n(D^n-1)+O(N^{n}).
\end{align*}
On the other hand,
 \begin{align}\label{h0NA}
n!\cdot  h^0(\tilde X,NA)&=N^nA^n+O(N^{n-1})\cr
&=N^n\ell^n (n+1)^nD^n-N^n\big( \ell (n+1)- 1\big)^{n} +O(N^{n-1}).
\end{align}
Hence, 
 \begin{align*}
\beta(A,E)\ge \frac{c \ell^n (n+1)^n(D^n-1)}{\ell^n (n+1)^n\big(D^n- (1-\frac{1}{\ell (n+1)})^n\big)}.
\end{align*}
 Since $D^n\ge 2$, we can take the rational number $c$ sufficiently close to $\sqrt[n]{2}$, and for $\ell$ sufficiently large, this implies that  $
\beta(A,E)>1.$

 Next, we estimate $\beta(A,\tilde D_i)$.  Note that 
\begin{align}\label{NAD}
 NA-m\tilde D_i \equiv  \big(N\ell (n+1)-m)\pi^*D_i-\big( N\ell (n+1)-m-N\big)E  
\end{align}
 which is nef and big for $m\le N\ell (n+1)-N$.   Since $\pi^*D_i$ and  $\pi^*D_i-E=\tilde D_i$ is nef, we have 
\begin{align*}
 n!\cdot  h^0(\tilde X,NA-m\tilde D_i)
 &=  \big(\large(N\ell (n+1)-m\large)\pi^*D_i-\large( N\ell (n+1)-m-N\large)E \big)^n+O(N^{n-1})\\
 &= \big( N\ell (n+1)- m\big)^nD^n-\big( N\ell (n+1)- m-N \big)^n+O(N^{n-1}).
\end{align*}
(See \cite[p.~233, Case $k\le n$]{Aut09}.)
Summing over $0\le m\le N\ell (n+1)-N$, we deduce
\begin{align*}
&n!\cdot \sum_{m=0}^{N\ell (n+1)-N} h^0(\tilde X,NA-m\tilde D_i) \cr
&\ge  \frac{N^{n+1}}{n+1}\big(\big(\ell^{n+1}(n+1)^{n+1}-1\big) D^n- (\ell (n+1)-1)^{n+1}\big)+O(N^{n}).
\end{align*}
Then by \eqref{h0NA}, we have
 \begin{align*}
\beta(A,\tilde D_i)\ge \frac{ \big(\ell^{n+1}(n+1)^{n+1}-1\big) D^n- (\ell (n+1)-1)^{n+1} }{ \ell^n (n+1)^{n+1} D^n- (n+1)(\ell (n+1)-1)^n\big)}.
\end{align*}
To show $\beta(A,\tilde D_i)>\ell$, we compare the numerator with $\ell$ times the denominator:
\begin{align*}
&(\ell^{n+1}(n+1)^{n+1}-1)D^n - (\ell(n+1)-1)^{n+1} \\
&\quad - \ell\Big(\ell^n (n+1)^{n+1}D^n - (n+1)(\ell(n+1)-1)^n\Big) \\
&= -D^n + (\ell(n+1)-1)^n > 0
\end{align*}
for $\ell$  sufficiently large. This completes the proof of the claim.
\end{proof}

\section{Applications}\label{Application}
  
\subsection{Family of Unit Equations}
This section presents the proofs of Theorems \ref{unit1} and \ref{unit2}, concerning a one-parameter family of unit equations in multiple variables, as applications of Theorems \ref{mainTheorem2} and \ref{mainTheoremCZrev}, respectively.

\subsubsection{\bf Proof of Theorem  \ref{unit1} (Application of Theorem \ref{mainTheorem2})}\label{Secunit1}
 
\begin{proof}[Proof of Theorem  \ref{unit1}]
The proof essentially follows the surface case \cite[Theorem 5.14]{HLX}. We provide the necessary details here for completeness.

As in the surface case, the statement of the theorem is independent of the ordering of $f_1,\ldots,f_{n+1}$. Thus, after a suitable permutation, we may assume that
\[
d_1\geq \cdots \geq d_{n+1}>0.
\]

Rather than studying the original equation directly, we consider the slightly modified equation:
\begin{align}
\label{funiteq2}
f_1(t)u_1+f_2(t)u_2^{d_1+1-d_2}+\cdots +f_{n}(t)u_n^{d_1+1-d_n}=f_{n+1}(t),
\end{align}
where $t\in \mathcal{O}_{k,S}, u_i \in \mathcal{O}_{k,S}^*$.
Since $\mathcal{O}_{k,S}^*$ is finitely generated, we can find a finite Galois extension  $L/k$ and a finite set of places $S'$ of $L$ such that every element of $\mathcal{O}_{k,S}^*$ has a $(d_1+1-d_i)$-th root for all $i \in \{2,\ldots,n\}$ in $\mathcal{O}_{L,S'}^*$. Then \eqref{funiteq} reduces to studying the equation \eqref{funiteq2} (with $(k,S)$ replaced by $(L,S')$).

Let $F_i\in k[x_0,x_1,\ldots,x_{n+1}]$ be the homogeneous polynomial defined by $F_i(x_0,\dots,x_{n+1})=f_i(x_0/x_{n+1})x_{n+1}^{d_i}$, $i=1,\ldots,n+1$. Let $X\subset \mathbb{P}^{n+1}$ be the hypersurface defined by 
\begin{align*}
F:=x_1F_1+x_2^{d_1+1-d_2}F_2+\cdots +x_{n}^{d_1+1-d_n}F_n-x_{n+1}^{d_1+1-d_{n+1}}F_{n+1}=0.
\end{align*}
Since the polynomials $f_i$ do not have a common zero, it follows easily that $F$ is irreducible in $\overline{k}[x_0,x_1,\cdots,x_{n+1}]$.

Let $H_i$ be the hyperplane of $\mathbb{P}^{n+1}$ defined by $x_i=0$,  and let $D_i=H_i|_X$ be the corresponding divisor on $X$ for $i=1,\cdots,n+1$.   
 Let $P_0=[1:0:\cdots:0]\in X(k)$.  By construction,
 $\bigcap_{i=1}^{n+1} D_i=\{P_0\}\neq \emptyset$, and  $D_1\sim \cdots \sim D_{n+1}$.  We set $D=D_1$.

Let $I \subset \{1,\ldots,n+1\}$ with $|I|=n$, then the hyperplanes $H_i$, $i \in I$, and the hypersurface $X$ are in general position on $\mathbb{P}^{n+1}$, since $F_i(1,0, \ldots,0) \neq 0$ for all $i$. 
Following \cite[Remark 2.13]{HLX}, the divisors $D_i$, $i \in I$, intersect properly on $X$.  From the equation for $F$, the local ideal at $P_0$ satisfies  $(x_{1},\ldots,x_{n}) \mathcal{O}_{P_0,X}=(x_{1},\cdots,x_{n},x_{{n+1}}^{d_1+1-d_{{n+1}}}) \mathcal{O}_{P_0,X}$. It follows that the local intersection multiplicity of $D_i$ (for  $1\le i\le n$) at the point  $P_0$ is 
\begin{align*}
(D_{1}\cdot\ldots\cdot D_{n})_{P_0}=d_1+1-d_{{n+1}}.
\end{align*}

Given $\deg X=d_1+1$, we have $D^n=(H_1|X)^n=(H_1)^n.X=d_1+1$. Applying Lemma 3.11 in \cite{HLX}, we obtain:
\begin{align*}
\beta(D,(\bigcap_{i=1}^{n} D_{i})_{P_0}) &\geq \frac{n}{n+1}\sqrt[\leftroot{-3}\uproot{20}n]{\frac{D^n}{(D_{1}\cdot\ldots\cdot D_{n})_{P_0}}}\\
&\geq  \frac{n}{n+1}\sqrt[\leftroot{-3}\uproot{20}n]{\frac{d_1+1}{d_1+1-d_{{n+1}}}} .
\end{align*}
 By assumption \eqref{degrees}, we have $\beta(D,(\bigcap_{i=1}^{n} D_{i})_{P_0})>1$. 
We are thus in a position to apply Theorem \ref{mainTheorem2}.   There exists a proper Zariski-closed subset $\tilde Z\subset X$ over $L$ such that for any set $\mathcal R\subset X(L)$ of $(\sum_{i=1}^{n+1}D_i, S')$-integral points, the set   $\mathcal R\setminus \tilde Z$ is finite.  
\end{proof}

\subsubsection{\bf Proof of Theorem  \ref{unit2} (Application of Theorem \ref{mainTheoremCZrev})}\label{Secunit2}
We will prove the following theorem, which implies Theorem  \ref{unit2}.
\begin{theorem}\label{UnitCZ}
Let $f=\sum_{i=1}^n x_i g_i$, where $g_i \in k[x_0,\ldots,x_n]$, $n\ge 3$, are homogeneous polynomials of the same degree such that $g_i(1,0,\ldots,0)\ne 0$, and  let $X$  be the hypersurface  in $\mathbb P^n$ defined by the zero locus of  $f=\sum_{i=1}^nx_ig_i$.   Then there exists a Zariski-closed subset $Z \subset X$ such that all but finitely many solutions of
\[
f(y,u_1,\ldots,u_n)=0, \quad (y,u_1,\ldots,u_n)\in \mathcal O_S\times \mathcal O_S^*\times\cdots\times \mathcal O_S^*
\]
are contained in $Z$.
\end{theorem}

\begin{proof}[Theorem \ref{UnitCZ} implies  Theorem  \ref{unit2}]
Let $g_i=x_{n+1}^d f_i(\frac{x_0}{x_{n+1}})$ for $1\le i\le n$ and let $g_{n+1}= -x_{n+1}^d f_{n+1}(\frac{x_0}{x_{n+1}}) $. These are homogeneous polynomials of degree $d$.  Furthermore, because each $f_i$ has degree exactly $d$, which implies
 $g_i(1,0,\hdots,0)\ne 0$ for all  $1\le i\le n+1$.   Let $X$  be the hypersurface  in $\mathbb P^{n+1}$ defined by the zero locus of  $F=\sum_{i=1}^{n+}x_ig_i $.  Theorem  \ref{unit2} then follows directly from Theorem \ref{UnitCZ} by considering the affine patch where $x_{n+1}=1$.
\end{proof}

\begin{proof}[Proof of Theorem \ref{UnitCZ}]
Let $H_1,\ldots,H_n$ be hyperplanes in $\mathbb{P}^{n}$ defined by the coordinate equations 
$x_1=0,\ldots,x_n=0$, respectively. Let $X$ be the hypersurface defined by $f=0$.  Let
\[
D_1:=H_1|_X,\ldots,D_n:=H_n|_X.
\]
Then $\bigcap_{i=1}^{n }D_i$ contains a single point $P=[1:0:\cdots:0].$ Since 
\[
 \frac{\partial f}{\partial x_j} 
= g_j+\sum_{i=1}^n x_i\,\frac{\partial g_i}{\partial x_j},
\]
at the point
$P=[1:0:\cdots:0],$ we have
$\frac{\partial f}{\partial x_j}(P)=g_j(1,0,\dots,0) \neq 0.$   Hence $X$ is nonsingular at the point $P$. Moreover, the tangent space to $D_i$ at $p$ is defined by the equation 
\[
\sum_{j=1}^n g_j(1,0,\dots,0)\, x_j=0 \ \quad\text{and} \quad  x_i=0.
\]
Since $g_j(1,0,\dots,0)\neq 0$ for every $1\le j\le n$, any $n-1$ of the divisors $D_i$ ($1\le i\le n$) intersect transversally at $P$.
We now show that $\bigcap_{i=2}^n D_i$
contains more than one point. Note that
\[
\bigcap_{i=2}^n D_i
=
\left(\bigcap_{i=2}^n  H_i\right)\cap [x_1g_1=0].
\]
Since $g_1(1,0,\dots,0)\neq 0$, we have
\[
P \notin \bigcap_{i=2}^n [x_i=0]\cap [g_1=0].
\]
 Consequently, $\bigcap_{i=2}^n D_i$ contains more than one point. 
This implies that $D_i^n \ge 2$ for each $i$. 
We are now in a position to apply Theorem~\ref{mainTheoremCZrev} to complete the proof.

\end{proof}


\subsection{Application of Theorem~\ref{MainThmgcd}}\label{applyGCD}
We offer two applications of Theorem~\ref{MainThmgcd} by studying perfect powers of exponential Diophantine equations and of digits.

\subsubsection{\bf Diophantine Equation $f(a_1^{m},\hdots,  a_{\ell}^{m},y)=b^n$}\label{applyGCDdiop}
 In this section, we consider exponential Diophantine equations of the form
\[
f(a_1^{m},\ldots,a_{\ell}^{m},y)=b^n,
\]
where $a_1,\ldots,a_{\ell},b>1$ are integers and $\gcd(a_1,\ldots,a_{\ell},b) > 1$. Our result extends the work of \cite[Theorem 5.7]{HLX} and serves as  our first application of Theorem~\ref{MainThmgcd}. The proof follows the general strategy of \cite{HLX}, with suitable modifications to accommodate the higher-dimensional setting.

\begin{theorem}\label{DiophantinrEq}
Let $F (X_1,\ldots,X_{\ell}, Y, Z) \in \mathbb{Q}[X_1,\ldots,X_{\ell}, Y, Z]$ be a homogeneous polynomial of degree $d \ge 2$, and let $a_1,\ldots,a_{\ell},b > 1$ be integers. Suppose that 
\begin{enumerate}
\item[\rm(i)] $F(0,\ldots,0,1,0) \neq 0$.

\item[\rm(ii)] For all $i \in \{1,\ldots,\ell\}$, neither $F(0,\ldots,0,Y,Z)$ nor $F(0,\ldots,X_i,0,\ldots,Y,0)$ is a power of a linear form in $\bar{\mathbb{Q}}[X_1,\ldots,X_{\ell},Y,Z]$.
\label{tancond}

\item[\rm(iii)] $\gcd(a_1,\ldots,a_{\ell},b) > 1$.
\end{enumerate}

Define 
\[
f(X_1,\ldots,X_{\ell},Y)=F(X_1,\ldots,X_{\ell},Y,1).
\]
Then there exists a proper Zariski closed subset $\tilde Z\subset \mathbb A^{\ell+1}$ such that the set of solutions
\[
\{(a_1^{m},\ldots, a_{\ell}^{m},y) \in \mathbb{Z}^{\ell+1} 
\mid 
f(a_1^{m},\ldots,a_{\ell}^{m},y) = b^n 
\text{ for some } n \ge m \ge 0\}
\]
is contained in $\tilde Z$, except for possibly finitely many points.
\end{theorem}

\begin{remark}
\begin{enumerate}
\item
A classic example of such a polynomial is a generalized Fermat-type form. For any degree $d\ge 2$, consider: $F(X_1,\ldots,X_{\ell},Y,Z)=Y^d+Z^d+X_1^d+\ldots+X_{\ell}^d$, where $d\ge 2$.
\item
When $\ell=1$, the condition $n \geq m$ can be removed, and the statement becomes identical to \cite[Theorem 5.7]{HLX}.
\end{enumerate}
\end{remark}

\begin{proof}
We first make a reduction.
Let $V_i$ be the hypersurface of $\mathbb{P}^{\ell+2}$ defined by
\[
b^iW^d = F(X_1,\ldots,X_{\ell},Y,Z),  \quad i\in \{0,\ldots, d-1\}.
\]
If $f(a_1^{m},\ldots, a_{\ell}^{m},y)=b^n$ for some $m,n,y \in \mathbb{Z}$, write
$n=dn_i+i$ with $n_i\in\mathbb{Z}$ and $i\in \{0,\ldots, d-1\}$. Then
$b^n=(b^{n_i})^d b^i$, and hence
\[
[W:X_1:\cdots:X_{\ell}:Y:Z]
=
[b^{n_i}:a_1^{m}:\cdots:a_{\ell}^{m}:y:1]
\in V_i(\mathbb{Q}).
\]

Let
\[
R_i
:=
\left\{
[b^{n_i}:a_1^{m}:\cdots:a_{\ell}^{m}:y:1]
\in V_i(\mathbb{Q})
\ \middle|\
m,n_i,y\in\mathbb{Z},\ m,n_i\ge 0
\right\}.
\]

Suppose that for each $V_i$ there exists a Zariski closed subset $Z_i \subset V_i$
such that $R_i \setminus Z_i$ is finite.
Let
\[
\pi_i : V_i \longrightarrow \mathbb{P}^{\ell+1},
\quad
[W:X_1:\cdots:X_{\ell}:Y:Z]
\longmapsto
[X_1:\cdots:X_{\ell}:Y:Z]
\]
be the projection onto the last $\ell+2$ coordinates.
Denote by $\widetilde{Z_i}$ the image $\pi_i(Z_i)$.
Since $V_i$ is projective, $\pi_i$ is proper, and therefore $\widetilde{Z_i}$ is a
Zariski closed subset of $\mathbb{P}^{\ell+1}$.

Let $\mathbb{A}^{\ell+1}$ be the affine open subset of $\mathbb{P}^{\ell+1}$ defined by $Z\neq 0$.
Then the conclusion of the theorem follows by taking
\[
\tilde Z:=\left(\bigcup_{i=0}^{d-1} \widetilde{Z_i}\right)\cap \mathbb{A}^{\ell+1}.
\]
Hence it suffices to show that, for each fixed $i$,  there exists a Zariski closed subset $Z_i \subset V_i$
such that $R_i \setminus Z_i$ is finite.
For notational simplicity, we write $V:=V_i$ and $R:=R_i$.
\[
R 
:=
\left\{
[b^{n}:a_1^{m}:\cdots:a_{\ell}^{m}:y:1]
\in V(\mathbb{Q})
\ \middle|\
m,n,y\in\mathbb{Z},\ m,n\ge 0
\right\}.
\]

 Let $D_1,\ldots,D_{\ell+2}$ be the divisors on $V$ defined by
\[
W=0,\quad X_i=0 \ (1\le i\le \ell), \quad \text{and } Z=0,
\]
respectively. Clearly the $D_i$ are linearly equivalent ample effective Cartier divisors; let $D$ be any divisor in this linear equivalence class. 

Since $F(0,\ldots,0,1,0) \neq 0$,  $V$ and  the hyperplanes defined by $W=0$, $X_i=0$ $(1\le i\le \ell)$, and $Z=0$ are in general position   in $\mathbb{P}^{\ell+2}$. By \cite[Remark 2.13]{HLX}, it follows that $D_1,\ldots,D_{\ell+2}$ intersect properly on $V$.

For $1\le j\le \ell+2$, let $I_j$ be the subset of $\{1,\ldots,\ell+2\}$ obtained by omitting the index $j$. By condition (ii), the sets 
\[
\Supp\!\left(\bigcap_{i\in I_{\ell+2}} D_i\right)(\bar{\mathbb{Q}})
\quad\text{and}\quad
\Supp\!\left(\bigcap_{i\in I_j} D_i\right)(\bar{\mathbb{Q}})
\]
for $2\le j\le \ell+1$ each contain more than one point. Moreover, the conditions $F(0,\ldots,0,1,0)\neq 0$ and $d\ge 2$ imply that 
\[
\Supp\!\left(\bigcap_{i\in I_1} D_i\right)(\bar{\mathbb{Q}})
\]
also contains more than one point.

Let $S$ be the set of places of $\mathbb{Q}$ defined by
\[
S=\{p \text{ prime} : p \mid b\prod_i a_i\}\cup\{\infty\}.
\]
Then $R$ is a set of $(D_1+\cdots+D_{\ell+2},S)$-integral points in $V(\mathbb{Q})$.
Let $P=[b^n:a_1^{m}:\cdots:a_\ell^{m}:y:1]\in R$, and  
\[
h:=h(P)=\log\max\{|b^n|,|a_1^{m}|,\ldots,|a_{\ell}^{m}|,|y|\}=h_D(P)+O(1).
\]
Let $p$ be any prime dividing $\gcd(a_1,\ldots,a_\ell,b)$. Using the standard local height functions for the divisors $D_i$ at the places $\infty$ and $p$ of $\mathbb{Q}$, we obtain
\begin{center}
\begin{tabular}{p{5cm}p{5cm}p{5cm}}
Divisor $D_i$ & $\lambda_{D_i,\infty}(P)$ & $\lambda_{D_i,p}(P)$ \\
$D_1$ & $h-n\log b$ & $n(\log p)\operatorname{ord}_p b$ \\
$D_i,\; 2\le i\le \ell+1$ & $h-m\log a_{i-1}$ & $m(\log p)\operatorname{ord}_p a_{i-1}$ \\
$D_{\ell+2}$ & $h$ & $0$
\end{tabular}
\end{center}

Let $\varepsilon$ be a sufficiently small positive number to be chosen later. By Theorem~\ref{MainThmgcd}, there exists a proper Zariski closed subset $\tilde Z\subset V$ such that for every $I_j$ and every $v\in S$,
\begin{align}
\min_{i\in I_j} \lambda_{D_i,v}(P)\le \varepsilon h_D(P), \label{equgcd}
\end{align}
for all but finitely many $P\in R\setminus Z$.

Applying \eqref{equgcd} with $(I_j,v)=(I_j,\infty)$, for $j=1,\ldots, \ell+1$, and $(I_{\ell+2},p)$, we obtain that for all $P=[b^n:a_1^{m}:\cdots:a_\ell^{m}:y:1]\in R\setminus \tilde Z$, then we have 
\begin{align}\label{nmbound} 
m\max_{1\le i\le\ell} \{\log a_i\} &\ge (1-\varepsilon)h_D(P)+O(1),\cr
 \max\{n\log b,\, m\max_{i\neq 1}\log a_i\} &\ge (1-\varepsilon)h_D(P)+O(1),\cr
 &\vdots\cr
 \max\{n\log b,\, m\max_{i\neq \ell}\log a_i\} &\ge (1-\varepsilon)h_D(P)+O(1),\cr
m=\min\{n,m\}&\le 2\varepsilon h_D(P)+O(1).
\end{align}
Taking
\[
0<\varepsilon<\frac{1}{2\max_i\{\log a_i\}+1},
\]
the first and last inequalities in \eqref{nmbound} imply that $h_D(P)$ is bounded for $P\in R\setminus \tilde  Z$. Since $D$ is ample, we conclude that $R\setminus \tilde Z$ is finite.

Finally, we note that when $\ell=1$, the assumption $n\ge m$ can be removed.   Indeed,  in the case of $n\le m$, the conclusion follows from the second and the last inequalities in \eqref{nmbound}, which imply that 
\[
n\le 2\varepsilon h_D(P)+O(1),\quad \text{and}\quad n\log b \ge (1-\varepsilon)h_D(P)+O(1).
\]
\end{proof}

\subsubsection{\bf Perfect Powers of Digits}\label{applyGCDdigit}

In this subsection, we revisit the problem of representing perfect powers with few nonzero digits   as an application of Theorem \ref{MainThmgcd}. 
As discussed in the introduction, Corvaja and Zannier \cite{CZ2013} bounded the number of perfect powers with exactly four nonzero binary digits, relying on a result from \cite{CZ2005} derived from the Subspace Theorem. In this subsection, we revisit this problem to show that Theorem \ref{MainThmgcd} can serve as a direct replacement for that Diophantine ingredient.  As in \cite{CZ2013}, our method yields more general statements; however, to keep the underlying ideas clear, we restrict ourselves to the following formulation.
 \begin{theorem}\label{digit}
Let $n \ge 2$ be a positive integer, let $d \ge 2$ be an integer, and let $0 < \alpha, \gamma < 1$ be two fixed real numbers. 
Then there exists a non-trivial polynomial $G \in \mathbb{Q}[X_1, \dots, X_n]$ such that
$G(2^{m_1}, \dots, 2^{m_n}) = 0$ for  all but finitely many  elements in the set 
\[
\mathcal{S} = \left\{ (2^{m_1}, \dots, 2^{m_n}) \in \mathbb{Z}^n \;\middle|\;
\begin{aligned}
& 0 < m_1 < \dots < m_n, \\
&m_{n-1} \le \alpha m_n \text{ or } m_1 \ge \gamma m_n, \\
& \text{and } 1 + 2^{m_1} + \dots + 2^{m_n} =z^d \text{ for some } z\in\mathbb{Z} 
\end{aligned}
\right\}.
\]

\end{theorem}

The following remark explains how the condition in the above theorem relates to the assumptions in \cite[Theorem~1]{CZ2005}.

\begin{remark}
Let $\nu$ be an absolute value of $\mathbb Q$.  
Let $\mathbf x_t=(x_{t1},\ldots,x_{tn})$,  be an infinite set of $n$-tuples in ${\mathbb Q^*}^n$ tending to zero in ${\mathbb Q_{\nu}}^n$.  
Recall the following condition appearing in \cite[Theorem~1]{CZ2005}:
\begin{align}\label{dominant}
{\hat h}(\mathbf x_t):=h(1,x_{t1},\ldots,x_{tn})
=O\!\left(-\log\!\left(\max_{1\le i\le n}\{|x_{ti}|_{\nu}\}\right)\right).
\end{align}

Let $P_m=(2^{m_1},\ldots,2^{m_n})$.  
When $\nu$ is the archimedean absolute value, we consider
\[
{\bf x}_m := (2^{m_0-m_n},\, 2^{m_1-m_n},\ldots,2^{m_{n-1}-m_n}),
\]
where $m_0:=0$. 
Then 
\[
\hat h({\mathbf x}_m)=\hat h(P_m)= m_n\log 2,
\]
and 
\[
-\log\!\left(\max_{1\le i\le n}\{|2^{m_{i-1}-m_n}|\}\right)
=(m_n-m_{n-1})\log 2.
\]

The condition $m_{n-1}\le\alpha m_n$ with $0<\alpha<1$ implies that 
$m_n-m_{n-1}\ge (1-\alpha)m_n$. Consequently,  condition \eqref{dominant} is satisfied.

Conversely, if \eqref{dominant} holds, then there exist positive real constants $c_1$ and $c_2$ such that 
\[
c_1(m_n-m_{n-1})\le m_n\le c_2(m_n-m_{n-1}).
\]
Clearly $c_2>1$, and the second inequality implies that 
\[
m_{n-1}\le\frac{c_2-1}{c_2} m_n.
\]
Therefore, our condition $m_{n-1}\le\alpha m_n$ with $0<\alpha<1$ is equivalent to \eqref{dominant} when $\nu$ is the archimedean absolute value on $\mathbb Q$.

Similarly, the condition $m_1\ge\gamma m_n$ with $0<\gamma<1$ is equivalent to \eqref{dominant} when $\nu$ is the $2$-adic absolute value on $\mathbb Q$, by taking ${\bf x}_m=P_m$.
\end{remark}

\begin{proof}
Let $k=\mathbb{Q}(\sqrt[d]{2},\mu_d)$, where $\mu_d$ is a primitive $d$-th root of unity. 
Let $S\subset M_k$ be the finite set of places containing the archimedean ones and the places lying above the finite place $2$. 
Let $H_0,\ldots,H_n$ be hyperplanes in $\mathbb{P}^{n+1}$ defined by the coordinate equations 
$x_0=0,\ldots,x_n=0$, respectively. 
Let $V$ be the hypersurface in $\mathbb{P}^{n+1}$ defined by
\[
x_{n+1}^d-x_0^d-x_1^d-\cdots-x_n^d=0 .
\]
Note that $H_0,\ldots,H_n,V$ are in general position in $\mathbb{P}^{n+1}$. 
Then, by Remark 2.13 in \cite{HLX}, the divisors
\[
D_0:=H_0|_V,\ldots,D_n:=H_n|_V
\]
intersect properly as divisors on $V$.

Consider the points
\[
P_m= [1:2^{\frac{m_1}{d}}:\cdots:2^{\frac{m_n}{d}}:y_m]\in V\subset \mathbb{P}^{n+1}(k),
\]
where $y_m\in\mathbb{Z}$ satisfies
\[
y_m^d=1+2^{m_1}+\cdots+2^{m_n}.
\]
A straightforward calculation shows that such points $P_m$ are $(D_0+\cdots+D_n,S)$-integral points on $V$.

Note that for any $n$ divisors among $D_0,\ldots,D_n$, their intersection contains more than one point  in $V$ (since $d\ge2$). 
Hence the intersection condition in Theorem \ref{MainThmgcd} is satisfied for any index set 
$I\subset\{0,\ldots,n\}$ with $|I|=n$. 
Denote
\[
I_j=\{0,\ldots,n\}\setminus\{j\}.
\]

Since $D$ is an ample divisor, we have $h_D(P_m)=O(h(P_m))$. 
Applying Theorem \ref{MainThmgcd} with
\[
\varepsilon=\min\!\left\{\frac{1-\alpha}{2},\frac{\gamma}{2}\right\},
\]
we obtain a proper Zariski closed subset $Z$ such that
\begin{equation}\label{gcddigit}
h_{\bigcap_{i\in I_j}D_i}(P)<\varepsilon\,h(P)
\end{equation}
for all $P\in V(k)\setminus Z$.

We first note that
\begin{equation}\label{heightP}
h(P_m)=\frac{1}{d}\log y_m^d,
\end{equation}
and
\begin{equation}\label{heightineq}
2^{m_n}\le y_m^d=1+2^{m_1}+\cdots+2^{m_n}\le n\cdot2^{m_n}.
\end{equation}
Therefore,
\begin{equation}\label{heightP2}
\frac{m_n}{d}\log2\le h(P_m)\le \frac{m_n}{d}\log2+\frac{1}{d}\log n .
\end{equation}

We now compute the gcd heights. By the integrality, we only need to consider the contributions from the archimedean place and the place above $2$. 
Using the standard local height for $D_i$ and an absolute value $\nu$ of $k$, we have
\[
\lambda_{D_i,\nu}(P_m)
=\frac{1}{d}\log\max\{|y_m^d|_\nu,1,|2^{m_1}|_\nu,\ldots,|2^{m_n}|_\nu\}
-\frac{1}{d}\log|2^{m_i}|_\nu,
\]
for $0\le i\le n$, where we set $m_0:=0$.

With this convention,
\begin{equation}\label{localhtinf}
\lambda_{\bigcap_{i\in I_j}D_i,\infty}(P_m)
=
\begin{cases}
\displaystyle \frac{1}{d}\log y_m^d-\frac{1}{d}m_n\log2,
& 0\le j\le n-1,\\[6pt]
\displaystyle \frac{1}{d}\log y_m^d-\frac{1}{d}m_{n-1}\log2,
& j=n ,
\end{cases}
\end{equation}
and
\begin{equation}\label{localht2adic}
\lambda_{\bigcap_{i\in I_j}D_i,2}(P_m)
=
\begin{cases}
0, & 1\le j\le n,\\[4pt]
\displaystyle \frac{1}{d}m_1\log2, & j=0 .
\end{cases}
\end{equation}

Combining these with \eqref{heightP} and \eqref{heightP2}, we obtain
\begin{equation}\label{localht2}
h_{\bigcap_{i\in I_0}D_i}(P_m)
= h(P_m)-\frac{1}{d}(m_n-m_1)\log2
\ge \frac{m_1}{d}\log2 ,
\end{equation}
and
\begin{equation}\label{localht3}
h_{\bigcap_{i\in I_n}D_i}(P_m)
= h(P_m)-\frac{1}{d}m_{n-1}\log2
\ge \frac{1}{d}(m_n-m_{n-1})\log2 .
\end{equation}

Thus, for all $P_m\notin Z$, from \eqref{gcddigit} and \eqref{heightP2} we obtain
\begin{equation}\label{localht4}
m_1<\varepsilon\!\left(m_n+\frac{\log n}{\log2}\right)
\end{equation}
and
\begin{equation}\label{localht5}
(1-\varepsilon)m_n<m_{n-1}+\varepsilon\frac{\log n}{\log2}.
\end{equation}

If $m_{n-1}\le \alpha m_n$, then our choice $\varepsilon\le\frac{1-\alpha}{2}$ together with \eqref{localht5} implies
\[
m_n<\frac{\log n}{\log2}.
\]
Similarly, if $m_1\ge \gamma m_n$, we obtain the same conclusion from \eqref{localht4} and the condition $\varepsilon\le\frac{\gamma}{2}$.
Therefore, there are only finitely many such points $P_m\in V(k)\setminus Z$.
 
Finally, for the points $P_m=[y_m:1:2^{m_1/d}:\cdots:2^{m_n/d}] \in Z$, we claim that there exists a non-trivial polynomial $G \in \mathbb{Q}[X_1, \dots, X_n]$ such that $G(2^{m_1}, \dots, 2^{m_n})=0$. 
To see this, consider the projection onto the last $n+1$ coordinates 
$\pi  : V  \rightarrow \mathbb{P}^{n}$.  Then $\pi(Z)$ is a proper Zariski closed subset of $\mathbb P^n$.  Therefore, 
there exists a non-trivial polynomial $R \in k[x_1, \dots, x_n]$ such that $R(2^{m_1/d}, \dots, 2^{m_n/d}) = 0$ for all points $P_m=[y_m:1:2^{m_1/d}:\cdots:2^{m_n/d}] \in Z$. 

Next, we define the Galois norm of $R$:
\[ P(X_1, \dots, X_n) = \prod_{\sigma \in \text{Gal}(k/\mathbb{Q})} R^\sigma(X_1, \dots, X_n) \in \mathbb{Q}[X_1, \dots, X_n] \]
To ensure the polynomial is defined over the $d$-th powers of the variables, we take the product:
\[ G(X_1^d, \dots, X_n^d) = \prod_{k_1=0}^{d-1} \dots \prod_{k_n=0}^{d-1} P(\zeta_d^{k_1}X_1, \dots, \zeta_d^{k_n}X_n). \]
We note that the right-hand side is invariant under $X_i \mapsto \zeta_d X_i$, ensuring that $G$ is a polynomial in $X_1^d, \dots, X_n^d$. Substituting $X_i = 2^{m_i/d}$, we conclude that $G(2^{m_1}, \dots, 2^{m_n}) = 0$.
\end{proof}


\begin{thebibliography}{11}

\bibitem{Aut09}
\textsc{P. Autissier}, \textit{G\'eom\'etries, points entiers et courbes enti\'eres}, Ann. Sci. \' Ec.
Norm. Sup\'er {\bf 42} (2009), no.~4, 221--239.
 \bibitem{Aut11}
\textsc{P. Autissier}, \textit{Sur la non-densit\'e{} des points entiers}, Duke Math. J. {\bf 158} (2011), no.~1, 13--27.

   \bibitem{CZ00}
 \textsc{P. Corvaja and P. Zannier,} 
 \emph{On the Diophantine equation $f(a^m,y)=b^n$}, Acta Arith, \textbf{94}(2000), no. 1, 25--40.
\bibitem{CZ2004}  
 \textsc{P. Corvaja and P. Zannier,} 
 \emph{On the integral points on  surfaces}, Annal of Mathematics,  2nd series {\bf 160}(2004), no. 2, 705--726.
 \bibitem{CZ2005}  
 \textsc{P. Corvaja and P. Zannier,} 
 \emph{ $S$-unit points on analytic hypersurfaces}, Ann. Sci. \'Ecole  Norm. Sup.,  {\bf 38} (2005), no.~1, 76--92.
 
\bibitem{CZ2006}  
 \textsc{P. Corvaja and P. Zannier,} 
 \emph{On the integral points on certain surfaces}, IMRN, 2006, Article ID 98623, 1--20.

\bibitem{CZ2010}  
 \textsc{P. Corvaja and P. Zannier,} 
 \emph{ Integral points, divisibility between values of polynomials and entire
curves on surfaces},  Adv. Math. {\bf 225} (2010), no. 2, 1095--1118.
 
   
  \bibitem{CZ2013}  
 \textsc{P. Corvaja and P. Zannier,} 
 \emph{ Finiteness of odd perfect powers with four nonzero binary digits},    
 Ann. Inst. Fourier (Grenoble), {\bf 63} (2013), no.~2, 715--731.
  


 \bibitem{GNSW} \textsc{J.~Guo, K.~D.~Nguyen, C.-L.~Sun and J.~T.-Y. Wang,} \emph{Vojta's abc Conjecture  for algebraic tori and applications over function fields}, Advances in Mathematics, Paper No. 110358, 37 pp., 2025. 

\bibitem{Har}
\textsc{R. Hartshorne}, \textit{Algebraic geometry}, Graduate Texts in Mathematics, No. 52, Springer, New York-Heidelberg, 1977.

\bibitem{HL2021}
\textsc{G.~Heier and A.~Levin},
\emph{A generalized Schmidt subspace theorem for closed subschemes},
  Amer. J. Math., {\bf 143}(1)(2021), 213--226.
  
\bibitem{HLX}
\textsc{K. Huang, A. Levin and Z. Xiao}, \textit{A new Diophantine approximation inequality on surfaces and its application}, arXiv:2406.18879v1

\bibitem{HL}
\textsc{K. Huang and A. Levin}, \textit{Greatest common divisors on the complement
of numerically parallel divisors},  arXiv:2207.14432

\bibitem{Levin06}
\textsc{A. Levin},\textit{One-parameter families of unit equations}, Math. Res. Lett. {\bf 13}(2006), no. 5-6, 935--945.

\bibitem{LevinGCD}
\textsc{A. Levin},\textit{Greatest common divisors and Vojta's conjecture for blowups of algebraic
tori}, Invent. Math. {\bf 215} (2019), no. 2, 493--533.

\bibitem{LW}
\textsc{A.~Levin and J.~T.-Y. Wang},
  Greatest common divisors of analytic functions and nevanlinna theory on
  algebraic tori,  \emph{J. Reine Angew. Math.}, \textbf{767}
  (2020), 77--107.


 \bibitem{MR2015}
 \textsc{D.~McKinnon and M.~Roth}, 
  \emph{Seshadri constants, diophantine approximation and Roth's theorem for arbitrary varieties}, 
  Invent. Math.   {\bf 200} (2015),  no. 2, 513--583.


\bibitem{NW}
\textsc{J. Noguchi and J. Winkelmann}, \textit{Nevanlinna theory in several complex variables and Diophantine approximation}, Grundlehren der mathematischen Wissenschaften, 350, Springer, Tokyo, 2014.


  \bibitem{NWY07}
  \textsc{J. Noguchi, J.Winkelmann and K. Yamanoi,} 
  \emph{Degeneracy of holomorphic curves into algebraic varieties},  J. Math. Pures  Appl, \textbf{88}(2007), 293--306.



\bibitem{RTW21}
\textsc{E. Rousseau, A. Turchet and J. T.-Y. Wang}, \textit{Nonspecial varieties and generalized Lang-Vojta conjectures}, Forum of Mathematics, Sigma, {\bf 9} (2021), Paper No. e11.

\bibitem{Rubook}
\textsc{M. Ru}, \textit{ Nevanlinna theory and its relation to Diophantine approximation}, World Scientific Publishing Co. Pte. Ltd., Hackensack, NJ, 2021.
\bibitem{RV20}
\textsc{M. Ru and P. Vojta}, \textit{A birational Nevanlinna constant and its consequences}, American Journal of Mathematics, {\bf 142} (2020), no. 3, 957--991.

\bibitem{RuWang2017}
\textsc{M.~Ru and J.~T.-Y.~Wang},
\emph{A subspace theorem for subvarieties},
  Algebra Number Theory, {\bf 11} (2017), no. 10, 2323--2337.

\bibitem{RuWang2022}
\textsc{M.~Ru and J.~T.-Y.~Wang},
\emph{The Ru-Vojta results for subvarieties},
  International Journal of Number Theory, {\bf 18} (2022), no. 1, 61--74. 
\bibitem{RW24}
\textsc{M. Ru  and J. T.-Y. Wang}, \textit{Campana's orbifold conjecture for numerically equivalent divisors}, Journal of Geometric Analysis, {\bf 35} (2025), article number 313.

\bibitem{Sil1987} \textsc{J.~H.~Silverman,} \emph{Arithmetic distance functions and height functions in Diophantine geometry,}
 Math. Ann. {\bf 279} (1987), no. 2, 193--216.

\bibitem{VojtaBook}
\textsc{P.~Vojta,} \emph{Diophantine approximation and Nevanlinna theory}, {\it Arithmetic Geometry,} 111-224, Lecture Notes in Mathematics {\bf 2009}, Springer-Verlag, Berlin, 2011.

\bibitem{Voj2023}  
 \textsc{P. Vojta,} 
 \emph{Birational Nevanlinna constants, beta constants, and diophantine approximation to closed subschemes}, J. Th\'eor. Nombres Bordeaux, {\bf 35} (2023), no. 1, 17--61.
 
 \bibitem{WX2025}
  \textsc{J. T.-Y. Wang and Z. Xiao,} 
 \emph{Hyperbolicity and GCD for $n+1$  divisors with non-empty intersection, arXiv:2506.03534}
 

\bibitem{WY}
\textsc{J. T.-Y. Wang and Y. Yasufuku}, \textit{Greatest common divisors of integral points of numerically equivalent divisors}, Algebra Number Theory {\bf 15} (2021), no. 1, 287--305.

\bibitem{Yamanoi} 
 \textsc{K. Yamanoi},  
\textit{ Kobayashi hyperbolicity and higher-dimensional Nevanlinna theory}, 
Progr. Math., 308
Birkh\"auser/Springer, Cham, 2015, 209--273.

\end{thebibliography}
\end{document}